\documentclass[12pt,english,british]{amsart}
\usepackage[T1]{fontenc}
\usepackage[latin9]{inputenc}
\usepackage{geometry}
\usepackage{color}
\usepackage{babel}
\usepackage{amstext}
\usepackage{amsthm}
\usepackage{amssymb}
\usepackage{stackrel}
\usepackage{setspace}
\usepackage[unicode=true,pdfusetitle,
 bookmarks=true,bookmarksnumbered=false,bookmarksopen=false,
 breaklinks=false,pdfborder={0 0 0},pdfborderstyle={},backref=false,colorlinks=true]
 {hyperref}

\makeatletter
\numberwithin{equation}{section}
\numberwithin{figure}{section}
\theoremstyle{plain}
\newtheorem{thm}{\protect\theoremname}[section]
\theoremstyle{definition}
\newtheorem{defn}[thm]{\protect\definitionname}
\theoremstyle{definition}
\newtheorem{example}[thm]{\protect\examplename}
\theoremstyle{plain}
\newtheorem{lem}[thm]{\protect\lemmaname}
\theoremstyle{plain}
\newtheorem{prop}[thm]{\protect\propositionname}
\theoremstyle{plain}
\newtheorem{cor}[thm]{\protect\corollaryname}

\makeatother

\addto\captionsbritish{\renewcommand{\corollaryname}{Corollary}}
\addto\captionsbritish{\renewcommand{\definitionname}{Definition}}
\addto\captionsbritish{\renewcommand{\examplename}{Example}}
\addto\captionsbritish{\renewcommand{\lemmaname}{Lemma}}
\addto\captionsbritish{\renewcommand{\propositionname}{Proposition}}
\addto\captionsbritish{\renewcommand{\theoremname}{Theorem}}
\addto\captionsenglish{\renewcommand{\corollaryname}{Corollary}}
\addto\captionsenglish{\renewcommand{\definitionname}{Definition}}
\addto\captionsenglish{\renewcommand{\examplename}{Example}}
\addto\captionsenglish{\renewcommand{\lemmaname}{Lemma}}
\addto\captionsenglish{\renewcommand{\propositionname}{Proposition}}
\addto\captionsenglish{\renewcommand{\theoremname}{Theorem}}
\providecommand{\corollaryname}{Corollary}
\providecommand{\definitionname}{Definition}
\providecommand{\examplename}{Example}
\providecommand{\lemmaname}{Lemma}
\providecommand{\propositionname}{Proposition}
\providecommand{\theoremname}{Theorem}

\begin{document}
\title[Simple Locally Finite Associative Algebras]{\singlespacing{} Local Systems of Simple Locally Finite Associative Algebras}
\author{}
\author{Hasan M. S. Shlaka}
\address{Department of Mathematics, Faculty of Computer Science and Mathematics,
University Kufa, Al-Najaf, Iraq.}
\email{hasan.shlaka@uokufa.edu.iq}
\begin{abstract}
In this paper, we study local systems of locally finite associative
algebras over fields of characteristic $p\ge0$. We describe the perfect
local systems and study the relation between them and their corresponding
locally finite associative algebras. $1$-perfect and conical local
systems are also be considered and described briefly\href{http://.}{.}
\end{abstract}

\maketitle
\begin{singlespace}

\section{Introduction}
\end{singlespace}

\begin{singlespace}
\global\long\def\ad{\operatorname{\rm ad}}%

\global\long\def\bbR{\mathbb{R}}%

\global\long\def\bbF{\mathbb{F}}%

\global\long\def\ccR{\mathcal{R}}%

\global\long\def\ccF{\mathcal{F}}%

\global\long\def\ccM{\mathcal{M}}%

\global\long\def\ccL{\mathcal{L}}%

\global\long\def\ccP{\mathcal{P}}%

\global\long\def\ccB{\mathcal{B}}%

\end{singlespace}

\selectlanguage{english}%
\global\long\def\dlim{\operatorname{\underrightarrow{{\rm lim}}}}%

\selectlanguage{british}%
\begin{singlespace}
\global\long\def\core{\operatorname{\rm core}}%

\global\long\def\dim{\operatorname{\rm dim}}%

\global\long\def\End{\operatorname{\rm End}}%

\global\long\def\gl{\operatorname{\rm \mathfrak{gl}}}%

\global\long\def\Im{\operatorname{\rm Im}}%

\global\long\def\ker{\operatorname{\rm ker}}%

\global\long\def\rad{\operatorname{\rm rad}}%

\global\long\def\rank{\operatorname{\rm rank}}%

\global\long\def\Range{\operatorname{\rm Range}}%

\global\long\def\sym{\operatorname{\rm sym}}%

\global\long\def\skew{\operatorname{\rm skew}}%

\global\long\def\sl{\operatorname{\rm \mathfrak{sl}}}%

\global\long\def\sp{\operatorname{\rm \mathfrak{sp}}}%

\global\long\def\so{\operatorname{\rm \mathfrak{so}}}%

\global\long\def\su{\operatorname{\rm \mathfrak{su^{*}}}}%

\global\long\def\u{\operatorname{\rm \mathfrak{u^{*}}}}%

\end{singlespace}

Throughout the paper the field $\mathbb{F}$ is algebraically closed
of characteristic $p\ge0$ and $A$ is an infinite (countable) dimensional
locally finite associative algebra over $\mathbb{F}$. Recall that
an algebra $A$ is called locally finite if every finite set of elements
is contained in a finite dimensional subalgebra of $A$ \cite{BahBavZal2004}.
$A$ is called \emph{locally} \emph{simple} if for any finite subset
$U$ of $A$, there is a finite dimensional simple subalgebra that
contains $U$. Note that we do not require $A$ to have an identity
element.

Locally finite Lie algebras were described by Bahturin and Strade
\cite{bahStr1995} in 1995. They described these algebras in terms
of local systems of locally finite Lie algebras. Recall that a system
of finite dimensional subalgebras $\{A_{\alpha}\}_{\alpha\in\varGamma}$
of an algebra $A$ over $\bbF$ is said to be a \emph{local} \emph{system}
\emph{for} $A$ if $A=\cup_{\alpha\in\varGamma}A_{\alpha}$ and for
each $\alpha,\beta\in\varGamma$, there is $\gamma\in\varGamma$ such
that $A_{\alpha},A_{\beta}\subseteq A_{\gamma}$. Bahturin and Strade
in \cite{bahStr1995Some} provided some examples to describe locally
finite Lie algebras. In several papers (see for example \cite{baranov1998diagonal},
\cite{baranov1998}, \cite{baranov2015simple} and \cite{baranov1999diagonal})
Baranov et. al. classified simple locally finite Lie algebras over
algebraically closed fields of characteristic zeros. They showed that
there are two classes of locally finite Lie algebras which have can
be characterized in many different ways. These are the simple diagonal
locally finite Lie algebras and the finitary simple Lie algebras.
Inner ideals of the finitary simple Lie algebras were studied by Fernandiz
Lopes, Garcia and Gomez Lozano in \cite{LopGarLoz2}, while inner
ideals of the other class were studied by Baranov and Rowley in \cite{BavRow2013}. 

In 2004, Bahturin, Baranov and Zalesski \cite{BahBavZal2004} studied
the simple locally finite associative algebras over algebraically
closed fields of zero characteristic. They highlighted the relation
between them and locally finite Lie algebras over algebraically closed
field of characteristic 0. They proved that simple Lie subalgebras
of locally finite associative ones are either finite dimensional or
isomorphic to the Lie algebra of skew symmetric elements of some (Type1)
involution simple locally finite associative algebras. 

Baranov in \cite{baranov2013} proved that there is a natural bijective
correspondence between such Lie algebras and locally involution simple
associative algebras over algebraically closed fields of any characteristic
$\ne2$. Thus, to classify locally finite Lie algebras, we need to
study locally finite associative algebras briefly. This requires a
good understanding of their local systems. 

In this paper, we study local systems of locally finite associative
algebras over fields of characteristic $p\ge0$. Some of the results
in this papers were mentioned in \cite{BahBavZal2004} in the case
when $p=0$. We start with some of the preliminaries in section two.
Section three is devoted to the study of local systems of locally
finite associative algebras.

In section four, we describe the perfect local systems and study the
relation between them and their corresponding locally finite associative
algebras. Finally, the $1$-perfect and conical local systems were
described.

\section{Preliminaries}

Recall that $\bbF$ is an algebraically closed field of characteristic
$p\ge0$. 
\begin{defn}
\label{def:-locally finite} \cite{BahBavZal2004} A locally finite
algebra is an algebra $A$ over a field $\bbF$ in which every finite
set of elements of $A$ is contained in a finite dimensional subalgebra
of $A$.
\end{defn}

As an example of locally finite associative algebra is the algebra
$M_{\infty}(\bbF)$ of infinite matrices with finite numbers of non-zero
entries, that is, 
\begin{equation}
M_{\infty}(\bbF)=\stackrel[n=1]{\infty}{\cup}M_{n}(\bbF),\label{eq:M infty}
\end{equation}
where the algebra $M_{n}(\bbF)$ can be embedded in $M_{(n+1)}(\bbF)$
by putting $M_{n}(\bbF)$ in the left upper hand corner and bordering
the last column and row by $0$'s.
\begin{defn}
\label{def:local system } Let $A$ be an algebra over a field $\bbF$.
A system of finite dimensional subalgebras $\{A_{\alpha}\}_{\alpha\in\varGamma}$
of $A$ is said to be a \emph{local} \emph{system} for $A$ if $A=\cup_{\alpha\in\varGamma}A_{\alpha}$
and for each $\alpha,\beta\in\varGamma$, there is $\gamma\in\varGamma$
such that $A_{\alpha},A_{\beta}\subseteq A_{\gamma}$.
\end{defn}

Put $\alpha\le\beta$ if $A_{\alpha}\subseteq A_{\beta}$. Then $\varGamma$
is a directed set and $A=\dlim A_{\alpha}$ is a direct limit of the
algebras $A_{\alpha}$. Recall that $A$ is called perfect algebra
if $AA=A.$
\begin{defn}
1. A local system is said to be \emph{perfect }(resp\emph{. simple,
semisimple, nilpotent, ... etc}) if it consists of perfect (simple,
semisimple, nilpotent, ... etc) algebras.

2. A locally finite algebra $A$ is called \emph{locally} \emph{perfect
}(resp\emph{. simple, semisimple, nilpotent, ... etc}) if it consists
a perfect (simple, semisimple, nilpotent, ... etc) local system. 
\end{defn}

\begin{example}
\label{exa:locally simple-1} Suppose that $A$ is locally simple.
Then there is a chain of simple subalgebras 
\[
A_{1}\subset A_{2}\subset A_{3}\subset\ldots
\]
of $A$ such that $A=\cup_{i=1}^{\infty}A_{i}$. We can view $A$
as the direct limit $\dlim A_{i}$ for the sequence 
\[
A_{1}\rightarrow A_{2}\rightarrow A_{3}\rightarrow\ldots
\]
of injective homomorphisms of finite dimensional simple associative
algebras $A_{i}$. Since $\bbF$ is algebraically closed, each $A_{i}$
can be identified with the algebra $M_{n_{i}}(\bbF)$ of all $n_{i}\times n_{i}$-matrices
over $\bbF$ for some $n_{i}$. Moreover, each embedding $A_{i}\rightarrow A_{i+1}$
can be written in the following matrix form 
\[
M\rightarrow diag(M,\ldots,M,0,\ldots,0),\,\,\,\,\,\,M\in M_{n_{i}}(\bbF).
\]
\end{example}

\section{Local Systems of Locally Finite Associative Algebras}
\begin{lem}
\label{lem:Local system} Let $A$ be a locally finite associative
algebra over $\bbF$ and let $\{A_{\alpha}\}_{\alpha\in\Gamma}$ be
a system of finitely generated subalgebras of $A$. Then $\{A_{\alpha}\}_{\alpha\in\Gamma}$
is a local system of $A$ if and only if for every finite dimensional
subspace $P$ of $A$ there exists $\beta\in\Gamma$ such that $P\subseteq A_{\beta}$.
\end{lem}

\begin{proof}
Suppose first that $\{A_{\alpha}\}_{\alpha\in\Gamma}$ is a local
system of $A$. Let $P$ be a finitely generated subalgebra of $A$
and let $\{p_{1},\ldots,p_{n}\}$ be a basis of $P$. Then $A=\cup_{\alpha\in\Gamma}A_{\alpha}$
and for each $1\le i\le n$, there is $A_{i}\in\{A_{\alpha}\}_{\alpha\in\Gamma}$
such that $p_{i}\in A_{i}$. Thus $P\subseteq A_{\beta}$ for some
$\beta\in\Gamma$, as required.

Suppose now that for every finite dimensional subspace $P$ of $A$
there exists $\beta\in\Gamma$ such that $P\subseteq A_{\beta}$.
We need to show that $\{A_{\alpha}\}_{\alpha\in\Gamma}$ is a local
system of $A$. Let $x\in A$. Then for every subspace $P_{x}$ generated
by $x$, there is $\beta\in\Gamma$ such that $P_{x}\subseteq A_{\beta}$,
so $A=\cup_{\alpha\in\Gamma}A_{\alpha}$. It remains to note that
for any $\alpha,\beta\in\Gamma$, there is $\gamma\in\Gamma$ such
that $A_{\alpha},A_{\beta}\subseteq A_{\gamma}$. 
\end{proof}
\begin{lem}
\label{lem:Sublocal} Suppose that $\{A_{\alpha}\}_{\alpha\in\Gamma}$
is a local system of a locally finite associative algebra $A$ over
$\bbF$. If $\Gamma=\cup_{i=1}^{r}\Gamma_{i}$, then $\{A_{\alpha}\}_{\alpha\in\Gamma_{k}}$
(for some $1\le k\le r$) is a local system of $A$.
\end{lem}

\begin{proof}
Suppose that $\Gamma=\cup_{i=1}^{r}\Gamma_{i}$. We may assume that
$\Gamma$ is a is a disjoint union of $\Gamma_{i}$ of $\Gamma$ (because
if it is not, then we can decompose it as a disjoint union of subsets).
We are going to prove by contradiction that $\{A_{\alpha}\}_{\alpha\in\Gamma_{k}}$,
for some $1\le k\le r$, is a local system of $A$. Assume to the
contrary that $\{A_{\alpha}\}_{\alpha\in\Gamma_{i}}$ is not a local
system of $A$ for each $1\le i\le r$. Then there is a finite dimensional
subspace $P_{i}$ of $A$ such that $P_{i}\notin\{A_{\alpha}\}_{\alpha\in\Gamma_{i}}$
for each $i$. Consider the subspace $P=\bigoplus_{i=1}^{r}P_{i}$
of $A$. Then $P$ is finite dimensional with $P\nsubseteq A_{\alpha}$
for all $\alpha\in\Gamma$ (because $\Gamma$ is a disjoint union
of the subsets $\Gamma_{i}$). Thus, $\{A_{\alpha}\}_{\alpha\in\Gamma}$
is not a local system of $A$, a contradiction.
\end{proof}
\begin{lem}
\label{lem:GammaP} Let $\{A_{\alpha}\}_{\alpha\in\Gamma}$ be a local
system of $A$ and let $P$ be a finite dimensional subspace of $A$.
Then $\{A_{\beta}\}_{\beta\in\Gamma_{P}}$ is a local system of $A$,
where $\Gamma_{P}=\{\beta\in\Gamma\mid P\subseteq A_{\beta}\}$.
\end{lem}

\begin{proof}
By Lemma \ref{lem:Local system}, there is $\beta\in\Gamma$ such
that $A_{\beta}\supset P$. Let $P_{1}$ be a finite dimensional subalgebra
of $A$. Then $P_{2}=P+P_{1}$ is a finite dimensional subalgebra
of $A$, so by Lemma \ref{lem:Local system}, There is $\gamma\in\Gamma$
such that $A_{\gamma}\supset P_{2}$ with $\gamma\ge\beta$. Continuing
with this process we get the set $\Gamma_{P}=\{\beta\in\Gamma\mid A_{\beta}\supset P\}\subset\Gamma$.
By Lemma \ref{lem:Sublocal}, $\{A_{\beta}\}_{\beta\in\Gamma_{P}}$
is a local system of $A$. 
\end{proof}
\begin{prop}
\label{prop:Ideal } Let $\{A_{\alpha}\}_{\alpha\in\Gamma}$ be a
local system of a simple locally finite associative algebra $A$ over
$\bbF$. The following hold. 

1. \cite{shlaka2020} Let $\{I_{\alpha}\}_{\alpha\in\Gamma}$ be a
system of ideals such that $I_{\alpha}$ is an ideal of $A_{\alpha}$
for each $\alpha\in\Gamma$. Then either $\cap_{\alpha\in\Gamma}I_{\alpha}=0$
or for every $k\in\Gamma$ or there exists $\beta_{k}\in\Gamma$ such
that $A_{k}\subseteq I_{\beta_{k}}$.

2. Suppose that $\Gamma_{P}=\{\beta\in\Gamma\mid P\subseteq A_{\beta}\}$.
Then $\{A_{\beta}^{P}\}_{\beta\in\Gamma_{P}}$ is a local system of
$A$, where $A_{\beta}^{P}$ is the ideal of $A_{\beta}$ that generated
by the algebra $P$ for all $\beta\in\Gamma_{P}$. 
\end{prop}

\begin{proof}
1. This is proved in \cite{shlaka2020}. For the proof see \cite[Proposition 4.5]{shlaka2020}.

2. By Lemma \ref{lem:GammaP}, $\{A_{\beta}\}_{\beta\in\Gamma_{P}}$
is a local system of $A$. Let $A_{\beta}^{P}$ be the ideal of $A_{\beta}$
that generated by $P$. We need to show that $\{A_{\beta}^{P}\}_{\beta\in\Gamma_{P}}$
is a local system of $A$. Since $P\subseteq\cap_{\beta\in\Gamma_{P}}A_{\beta}^{P},$
we have $\cap_{\beta\in\Gamma_{P}}A_{\beta}^{P}\ne0$, so by (1),
for each $\beta\in\Gamma_{P}$, there is $\gamma\in\Gamma_{P}$ such
that $A_{\beta}\subseteq A_{\gamma}^{P}$. Thus, $L=\cup_{\beta\in\Gamma_{P}}A_{\beta}^{P}$.
It remain to note that $A_{\beta}^{P},A_{\gamma}^{P}\subseteq A_{\zeta}^{P}$.
Indeed, for each $\beta,\gamma\in\Gamma_{P}$, there is $\zeta\in\Gamma_{P}$
such that $A_{\beta},A_{\gamma}\subseteq A_{\zeta}$ because $\{A_{\beta}\}_{\beta\in\Gamma_{P}}$
is a local system of $A$. Therefore, $\{A_{\beta}^{P}\}_{\beta\in\Gamma_{P}}$
is a local system of $A$, as required. 
\end{proof}
\begin{prop}
\label{prop:Gamma_s local}Let $A$ be a simple locally finite associative
algebra and let $\{A_{\alpha}\}_{\alpha\in\Gamma}$ be a local system
of $A$. Let $S$ be a non-zero finite dimensional subspace of $A$
and let $A_{\alpha}^{s}$ be the ideal of $A_{\alpha}$ generated
by $S$. Then $\{A_{\beta}^{s}\}_{\beta\in\Gamma_{s}}$ is a local
system of $A$, where $\Gamma_{s}=\{\beta\in\Gamma\mid S\subseteq A_{\beta}\}$.
\end{prop}

\begin{proof}
This follows directly from Proposition \ref{prop:Ideal }.
\end{proof}
Recall that an associative algebra $A$ is nilpotent if there is a
positive integer $n$ such that $A^{n}=0$. Put $A=A^{1}$ and $A^{i}=AA^{i-1}$
for all $i>1$. 
\begin{defn}
1. We say that an associative algebra $P$ over a field $\bbF$ is
\emph{residually} \emph{nilpotent} if $\cap_{i=1}^{\infty}P^{i}=0$. 

2. We say that a locally finite associative algebra $A$ is \emph{residually}
\emph{nilpotent} if every finitely generated subalgebra $P$ of $A$
is residually nilpotent associative algebra.
\end{defn}

\begin{thm}
\label{thm:Residually nilpotent} Let $A$ be a simple locally finite
associative algebra over $\bbF$. Then every local system $\{A_{\alpha}\}_{\alpha\in\Gamma}$
contains an algebra which is not residually nilpotent.
\end{thm}

\begin{proof}
Let $\{A_{\alpha}\}_{\alpha\in\Gamma}$ be a local system of $A$.
Consider the algebra $P=\bbF p$, where $p\in A$ is non-zero. Put
$\Gamma_{p}=\{\alpha\in\Gamma\mid A_{\alpha}\supset P\}$. By Lemma
\ref{lem:GammaP}, $\{A_{\alpha}\}_{\alpha\in\Gamma_{p}}$ is a local
system of $A$. Assume to the contrary that all $A_{\alpha}$ are
residually nilpotent for all $\alpha\in\Gamma$. Then for each $\alpha\in\Gamma_{p}$,
there is a positive integer $n$ such that $p\in A_{\beta}^{n}$ and
$p\notin A_{\beta}^{n+1}$, so we get system of ideals $\{A_{\alpha}^{n}\}_{\alpha\in\Gamma_{p}}$
with $\cap_{\alpha\in\Gamma_{P}}A_{\alpha}^{n}\ne0$. By Proposition
\ref{prop:Ideal }, for each $\alpha\in\Gamma_{p}$, there is $\beta\in\Gamma_{p}$
such that $A_{\alpha}\subseteq A_{\beta}^{n}$. Thus, $A_{\alpha}^{2}\subseteq A_{\beta}^{n+1}$,
so $p\notin A_{\alpha}^{2}$ for all $\alpha\in\Gamma$. Since $p\in A^{2}$
(because $A=A^{2}$ as $A$ is simple), there exists $\gamma\in\Gamma$
such that $p=\sum_{i=1}^{n}x_{i}y_{i}$ for some $x_{i},y_{i}\in A_{\gamma}$
($1\le i\le n$). Thus, $p\in A_{\gamma}^{2}$, a contradiction with
$p\notin A_{\alpha}^{2}$ for each $\alpha\in\Gamma$. Therefore,
every local system of $A$ must contain an algebra which is not residually
nilpotent, as required.
\end{proof}
\begin{cor}
\label{cor:Residualy nilpotent} No simple locally finite associative
algebra can be locally residually nilpotent.
\end{cor}

\begin{proof}
This follows from Theorem \ref{thm:Residually nilpotent}.
\end{proof}

\section{Perfect local Systems}

Recall that a local system is said to be \emph{perfect }(resp\emph{.
simple, semisimple, nilpotent, ... etc}) if it consists of perfect
(simple, semisimple, nilpotent, ... etc) algebras.
\begin{thm}
\label{thm:Perfect local system} Any simple locally finite associative
algebra posses a perfect local system.
\end{thm}

\begin{proof}
Let $\{A_{\alpha}\}_{\alpha\in\Gamma}$ be a local system of $A$.
Put 
\[
A_{\alpha}^{\infty}=\cap_{i=1}^{\infty}A_{\alpha}^{i}.
\]
 Then $A_{\alpha}^{\infty}$ is a perfect subalgebra of $A^{\infty}$.
Moreover, $\{A_{\alpha}^{\infty}\}_{\alpha\in\Gamma}$ is a local
system of the subalgebra $A^{\infty}=\cup_{\alpha\in\Gamma}A_{\alpha}^{\infty}$
of $A$. Indeed, if $\alpha,\beta\in\Gamma$, then there is $\gamma\in\Gamma$
such that $A_{\alpha},A_{\beta}\subseteq A_{\gamma}$, so $A_{\alpha}^{\infty},A_{\beta}^{\infty}\subseteq A_{\gamma}^{\infty}$.
Note that $A^{\infty}$ is an ideal of $A$ and $A^{\infty}\ne0$
(by Theorem \ref{thm:Residually nilpotent}), but $A$ is simple,
so $A^{\infty}=A$. Therefore, $A$ contains a perfect local system,
as required. 
\end{proof}
\begin{lem}
\label{lem:BavStr=000026BavRow} Let $A$ be a simple locally finite
associative algebra over $\bbF$. The following holds: 
\end{lem}

\begin{enumerate}
\item \emph{\cite{bahStr1995} Let $\{A_{\alpha}\}_{\alpha\in\Gamma}$ be
a local system of $A$. Then for every $\alpha\in\Gamma$, there is
$\zeta\in\Gamma$ such that $A_{\alpha}\subset A_{\zeta}$ and $A_{\alpha}\bigcap\rad A_{\zeta}=0$.}
\item \emph{\cite{BavRow2013} If $\{A_{\alpha}\}_{\alpha\in\Gamma}$ is
a perfect local system of $A$, then there exists $\alpha'\in\Gamma$
for each $\alpha\in\Gamma$ such that $\rad A_{\beta}\bigcap A_{\alpha}=0$
for all $\beta\ge\alpha'$.}
\end{enumerate}
\begin{proof}
1. Let $A_{\alpha}\in\{A_{\alpha}\}_{\alpha\in\Gamma}$. Suppose that
$\Gamma_{\alpha}=\{\beta\in\Gamma\mid A_{\alpha}\subseteq A_{\beta}\}$.
By Lemma \ref{lem:GammaP}, $\{A_{\beta}\}_{\beta\in\Gamma_{P}}$
is a local system of $A$. Let $\{\rad A_{\beta}\}_{\beta\in\Gamma_{\alpha}}$
be the system of the radicals of $\{A_{\beta}\}_{\beta\in\Gamma_{\alpha}}$
such that $\rad A_{\beta}$ is the radical of $A_{\beta}$. By Theorem
\ref{thm:Residually nilpotent}, not all members of $\{A_{\beta}\}_{\beta\in\Gamma_{P}}$
are nilpotent algebras, so there exists $\gamma\in\Gamma_{\alpha}$
such that $A_{\gamma}\nsubseteq\rad A_{\beta}$ for all $\beta\in\Gamma_{\alpha}$.
Thus, $\cap_{\beta\in\Gamma_{\alpha}}\rad A_{\beta}=0$ (by Proposition
\ref{prop:Ideal }). 

Now, we have $A_{\alpha}\cap\rad A_{\beta}\subseteq\rad A_{\alpha}$
for all $\beta\in\Gamma_{\alpha}$. Since $\dim(\rad A_{\beta})$
is finite, there exist $\beta_{1},\beta_{2},\ldots,\beta_{n}$ ($n\ge1$)
such that $A_{\alpha}\cap(\cap_{i=1}^{n}\rad A_{\beta_{i}})=0$. Now,
take any $\zeta\in\Gamma_{\alpha}$ such that $A_{\zeta}\supseteq A_{\beta_{1}},\ldots,A_{\beta_{n}}$.
Then $A_{\beta_{i}}\cap\rad A_{\zeta}\subseteq\rad A_{\beta_{i}}$
for all $1\le i\le n$, so 
\[
A_{\alpha}\cap\rad A_{\zeta}\subseteq A_{\beta_{i}}\cap\rad A_{\zeta}\subseteq\rad A_{\beta_{i}},\,\,\,\,\text{for all}\,\,\,1\le i\le n.
\]
Therefore, 
\[
A_{\alpha}\cap\rad A_{\zeta}\subseteq A_{\alpha}\cap(\cap_{i=1}^{n}\rad A_{\beta_{i}})=0,
\]
as required. 

2. This follows from part (1.).
\end{proof}
\begin{thm}
\label{thm:Maximal ideal} Let $\{A_{\alpha}\}_{\alpha\in\Gamma}$
be a local system of a simple locally finite associative algebra $A$
over $\bbF$. Then for each $\alpha\in\Gamma$, there is $\gamma\in\Gamma$
such that $A_{\alpha}\subset A_{\gamma}$ and $M_{\gamma}\cap A_{\alpha}=0$
for some maximal ideal $M_{\gamma}$ of $A_{\gamma}$.
\end{thm}

\begin{proof}
By Lemma \ref{lem:BavStr=000026BavRow}, $A$ has a perfect local
system. Suppose that $\{A_{\alpha}\}_{\alpha\in\Gamma}$ is a perfect
local system of $A$. Fix any $\alpha\in\Gamma$. By Lemma \ref{lem:BavStr=000026BavRow}(2),
there is $\zeta\in\Gamma$ such that 
\begin{equation}
A_{\alpha}\subset A_{\zeta}\hspace*{1em}\text{and}\hspace*{1em}A_{\alpha}\cap\rad A_{\zeta}=0.
\end{equation}
Let $\Gamma^{\zeta}=\{\beta\in\Gamma\mid\beta\ge\zeta\}$. Then by
Lemma \ref{lem:Sublocal}, 
\[
\{A_{\beta}\}_{\beta\in\Gamma^{\zeta}}=\{A_{\beta}\in\{A_{\alpha}\}_{\alpha\in\Gamma}\mid A_{\zeta}\subset A_{\beta}\}
\]
is a local system of $A$. Let $\{P_{\zeta_{1}},\ldots,P_{\zeta_{t}}\}$
and $\{M_{\zeta_{1}},\ldots,M_{\zeta_{r}}\}$ be the set of all minimal
perfect and maximal ideals of $A_{\zeta}$, respectively. Consider
two subsets $\Gamma_{l,k}^{\zeta}$ and $\Gamma_{0}^{\zeta}$ of $\Gamma^{\zeta}$
defined as follows: 
\[
\Gamma_{k,l}^{\zeta}=\{\beta\in\Gamma^{\zeta}\mid A_{\beta}\hspace{1em}\text{has an ideal}\hspace*{1em}I_{\beta}\hspace*{1em}\text{with}\hspace*{1em}P_{\zeta_{k}}\subset I_{\beta}\cap A_{\zeta}\subset M_{\zeta_{l}},1\le k\le t,1\le l\le r\};
\]
\[
\Gamma_{0}^{\zeta}=\{\beta\in\Gamma^{\zeta}\mid\text{any proper ideal}\hspace*{1em}I_{\beta}\hspace*{1em}\text{of}\hspace*{1em}A_{\beta}\hspace*{1em}\text{satisfies}\hspace*{1em}I_{\beta}\cap A_{\zeta}\subset\rad A_{\zeta}.
\]
Then $\Gamma^{\zeta}=\Gamma_{k,l}^{\zeta}\cup\Gamma_{0}^{\zeta}$,
so by Lemma \ref{lem:Sublocal}, either $\{A_{\beta}\}_{\beta\in\Gamma_{k,\ell}^{\zeta}}$
or $\{A_{\beta}\}_{\beta\in\Gamma_{0}^{\zeta}}$ is a local system
of $A$. We claim that $\{A_{\beta}\}_{\beta\in\Gamma_{0}^{\zeta}}$
is a local system of $A$. Assume to the contrary that $\{A_{\beta}\}_{\beta\in\Gamma_{k,l}^{\zeta}}$
is local system. Then for any $A_{\gamma}\in\{A_{\beta}\}_{\beta\in\Gamma_{k,l}^{\zeta}}$,
there is an ideal $I_{\gamma}\subset A_{\gamma}$ such that $P_{\zeta_{k}}\subset I_{\gamma}\cap A_{\zeta}\subseteq M_{\zeta_{l}}$
for some $k$ and $l$. Hence, we get a system of ideals $\{I_{\beta}\}_{\beta\in\Gamma_{k,\ell}^{\zeta}}$
such that $I_{\beta}$ is an ideal of $A_{\beta}$ with $P_{\zeta_{k}}\subset I_{\beta}\cap A_{\zeta}\subset M_{\zeta_{l}}$
for some $1\le k\le t$ and $1\le l\le r$. Note that I\_$\beta\nsupseteq A_{\zeta}$
for all $\beta\in\Gamma_{k,l}^{\zeta}$ (because if $A_{\zeta}\subseteq I_{\beta}$,
then $A_{\zeta}=I_{\beta}\cap A_{\zeta}\subset M_{\zeta_{l}}$, a
contradiction as $M_{\zeta_{l}}$ is a proper maximal ideal of $A_{\zeta}$.
Since $A_{\zeta}\subseteq A_{\beta}$, for each $\beta\in\Gamma_{k,l}^{\zeta}$,
there is no $\delta_{\beta}\in\Gamma_{k,l}^{\zeta}$ such that $A_{\beta}\subseteq I_{\delta_{\beta}}$,
so by Proposition \ref{prop:Ideal }, 
\[
\cap_{\beta\in\Gamma_{k,l}^{\zeta}}I_{\beta}=0\text{, \, but \ensuremath{\,\cap_{\beta\in\Gamma_{k,l}^{\zeta}}I_{\beta}\in\{I_{\beta}\}_{\beta\in\Gamma_{k,l}^{\zeta}}}},
\]
so there is some $1\le k\le t$ such that $0\ne P_{\zeta_{k}}\subset\cap_{\beta\in\Gamma_{k,l}}I_{\beta}$,
a contradiction.

Hence, $\{A_{\beta}\}_{\beta\in\Gamma_{0}^{\zeta}}$ is a local system
of $A$. Thus, for every proper ideal $I_{\gamma}$ of $A_{\gamma}\in\{A_{\beta}\}_{\beta\in\Gamma_{0}^{\zeta}}$,
we have $I_{\gamma}\cap A_{\zeta}\subseteq\rad A_{\zeta}$, so for
$\alpha\in\Gamma$, there is $\gamma\in\Gamma_{0}^{\zeta}\subset\Gamma^{\zeta}\subset\Gamma$
such that if $I_{\gamma}$ is a proper ideal of $A_{\gamma}$, then
\[
I_{\gamma}\cap A_{\alpha}=(M_{\gamma}\cap A_{\zeta})\cap A_{\alpha}\subset\rad A_{\zeta}\cap A_{\alpha}=0.
\]
Therefore, for each $\alpha\in\Gamma$, there exists $\gamma\in\Gamma$
such that $A_{\alpha}\subset A_{\gamma}$ and $M_{\gamma}\cap A_{\alpha}=0$
for some maximal ideal $M_{\gamma}$ of $A_{\gamma}$, as required.
\end{proof}
If $A$ is finite dimensional, then it follows by Wedderburn-Malcev
Theorem (see \cite[Theorem 1]{BavMudShk2018}) that there exists a
semisimple subalgebra $S$ of $A$ such that $A=S\oplus\rad A$ and
for any semisimple subalgebra $Q$ of $A$, there is $r\in\rad A$
with $Q\subseteq(1+r)S(1+r)$. 
\begin{thm}
\cite[Theorem 6]{BavMudShk2018} Let $A$ be a finite dimensional
algebra and let $I$ be a left ideal of $A$. Suppose that $A/R$
is separable. Then there exists a semisimple subalgebra $S$ of $A$
such that $A=S\oplus\rad A$ and $I=I_{S}\oplus I_{\rad A}$, where
$I_{S}=I\bigcap S$ and $I_{\rad A}=I\bigcap\rad A$. 
\end{thm}

Recall that the \emph{rank of a perfect finite dimensional algebra
}$A$ is the smallest rank of the simple components of $A/\rad A$. 
\begin{thm}
\label{thm:simple has perfect local} Every simple locally finite
associative algebra over $\bbF$ has a perfect local system of arbitrary
large rank.
\end{thm}

\begin{proof}
Let $A$ be a simple locally finite associative algebra. Then by Theorem
\ref{thm:Perfect local system}, $A$ has a perfect local system.
Let $\{A_{\alpha}\}_{\alpha\in\Gamma}$ be a perfect local system
of $A$. Then by Theorem \ref{thm:Maximal ideal}, for each $\alpha\in\Gamma$,
there is $\gamma\in\Gamma$ such that $A_{\alpha}\subset A_{\gamma}$
and $M_{\gamma}\cap A_{\alpha}=0$ for some maximal ideal $M_{\gamma}$
of $A_{\gamma}$. Since $A_{\gamma}$ is finite dimensional, $A_{\gamma}=S_{\gamma}\oplus R_{\gamma}$,
where $S_{\gamma}$ is a Levi subalgebra of $A_{\gamma}$ and $R_{\gamma}=\rad A_{\gamma}$
is the radical of $A_{\gamma}$. Let $\{S_{\gamma}^{1},\ldots,S_{\gamma}^{n}\}$
be the set of the simple components of $S_{\gamma}$. First we claim
that $R_{\gamma}\subseteq M_{\gamma}$. Assume to the contrary that
$R_{\gamma}\nsubseteq M_{\gamma}$. Then $M_{\gamma}+R_{\gamma}\ne M_{\gamma}$.
Since $M_{\gamma}$ is maximal, $A_{\gamma}=M_{\gamma}+R_{\gamma}$.
Thus, 
\[
A_{\gamma}/M_{\gamma}=(M_{\gamma}+R_{\gamma})/M_{\gamma}\cong R_{\gamma}/(R_{\gamma}\cap M_{\gamma})\ne0
\]
is a non-zero nilpotent quotient algebra of $A_{\gamma}$, but $A_{\gamma}^{2}=A_{\gamma}$
(as $A_{\gamma}$ are perfect for all $\gamma$), so $A_{\gamma}=A_{\gamma}^{n}\subseteq M_{\gamma}$,
a contradiction. Thus, $R_{\gamma}\subset M_{\gamma}$. Note that
$M_{\gamma}\nsupseteq S_{\gamma}$ (because $M_{\gamma}\ne A_{\gamma}$).
Since 
\[
A_{\gamma}/M_{\gamma}=(S_{\gamma}\oplus R_{\gamma})/M_{\gamma}=(S_{\gamma}+M_{\gamma})/M_{\gamma}\cong S_{\gamma}/(S_{\gamma}\cap M_{\gamma})\ne0,
\]
$A_{\gamma}/M_{\gamma}\cong S_{\gamma}^{i}$ for some $1\le i\le n$.
Since $A_{\alpha}\cap M_{\gamma}=0$, we have that $A_{\alpha}\subseteq S_{\gamma}^{i}$.
Therefore, there is a simple component $S_{\gamma}^{i}$ in every
Levi subalgebra $S_{\gamma}$ of $A_{\gamma}$ such that $\dim S_{\gamma}^{i}\ge\dim A_{\gamma}$. 
\end{proof}

\section{$1$-Perfect and Conical Local Systems}
\begin{defn}
\cite{BavShk} Let $A$ be an associative algebra over a field $\bbF$.
Then $A$ is called $1$-\emph{perfect} if $A$ has no ideals of codimension
$1$. 

An ideal $I$ of $A$ is called $1$-\emph{perfect} if as an algebra
$I$ is $1$-perfect. By using the second and the third Isomorphism
Theorems, we obtain the following well known properties.
\end{defn}

\begin{lem}
\label{1p} \cite{BavShk} (i) The sum of $1$-perfect ideals of an
associative algebra $A$ is a $1$-perfect ideal of $A$.

(ii) Let $P$ be a $1$-perfect ideal of $A$ and let $C$ be a $1$-perfect
ideal of $A/P$. Then the full preimage of $C$ in $A$ is $1$-perfect.
\end{lem}

By using Lemma \ref{1p}(i) we get that every associative algebra
contains the largest $1$-perfect ideal. 
\begin{defn}
\label{def:1-perfect radical} Let $A$ be an associative algebra
and let $\ccP_{A}$ be the largest $1$-perfect ideal of $A$. Then
$\ccP_{A}$ is said to be \emph{the }$1$\emph{-perfect radical }of
$A$.
\end{defn}

We will need the following simple fact.
\begin{lem}
\label{lem:ideal of sub is ideal} Let $P$ be an ideal of an associative
algebra $A$ over a field $\bbF$. If $P'$ is an ideal of $P$ with
$P'^{2}=P'$, then $P'$ is an ideal of $A$.
\end{lem}

The following results due to Baranov and Shlaka \cite{BavShk} Shows
that $\ccP_{A}$ has radical-like properties.
\begin{prop}
\label{1pproperties} Let $A$ be an associative algebra over $\bbF$.
The following hold.

(1) $\ccP_{A}^{2}=\ccP_{A}$.

(2) $\ccP_{\ccP_{A}}=\ccP_{A}$.

(3) $\ccP_{A/\ccP_{A}}=0$.

(4) Consider any maximal chain of subalgebras $A=A_{0}\supset A_{1}\supset\dots\supset A_{r}$
of $A$ such that $A_{i+1}$ is an ideal of $A_{i}$. If $\dim A_{i}/A_{i+1}=1$
for all $0\leq i\leq r-1$, then $A_{r}=\ccP_{A}$.
\end{prop}

\begin{thm}
\label{thm:1-perfect local system} Any simple locally finite associative
algebra posses a $1$-perfect local system.
\end{thm}

\begin{proof}
By Theorem \ref{thm:Perfect local system} $A$ contains a perfect
system. Let $\{A_{\alpha}\}_{\alpha\in\Gamma}$ be a perfect system
of $A$. Consider the largest $1$-perfect radical $\ccP_{A_{\alpha}}$
of $A_{\alpha}$ for each $\alpha\in\Gamma$. Since $\mathcal{P}_{A_{\alpha}}$
is finite dimensional, $\mathcal{P}_{A_{\alpha}}=S_{\alpha}\oplus\rad\mathcal{P}_{A_{\alpha}}$
for some Levi subalgebra $S_{\alpha}$ of $\mathcal{P}_{A_{\alpha}}$.
Let $\{S_{\alpha}^{1},\ldots,S_{\alpha}^{n}\}$ be the set of the
simple components of $S_{\alpha}$. We denote by $A_{\alpha}^{i}$
to be the ideal of $\mathcal{P}_{A_{\alpha}}$ generated by $S_{\alpha}^{i}$.
Fix any index, say $1$. Since $A_{\alpha}^{1}$ is perfect, by Lemma
\ref{lem:ideal of sub is ideal}, $A_{\alpha}^{1}$ is an ideal of
$A_{\alpha}$. Thus, by Proposition \ref{prop:Gamma_s local}, $\{A_{\alpha}^{1}\}_{\alpha\in\Gamma_{S_{\alpha}^{1}}}$
is a $1$-perfect local system of $A$, where $\Gamma_{S_{\alpha}^{1}}=\{\beta\in\Gamma\mid A_{\beta}\supset A_{\alpha}^{1}\}$. 
\end{proof}
\begin{defn}
\label{def:Conical} Let $\{A_{\alpha}\}_{\alpha\in\varGamma}$ be
a perfect local system of $A$. Then $\{A_{\alpha}\}_{\alpha\in\varGamma}$
is said to be \emph{conical} if $\varGamma$ contains a minimal element
$1$ such that

1. $A_{1}\subseteq A_{\alpha}$ for all $\alpha\in\varGamma$;

2. $A_{1}$ is simple;

3. for each $\alpha\in\varGamma$ the restriction of any natural $A_{\alpha}$-module
to $A_{1}$ has a non-trivial composition factor. 
\end{defn}

We denote by the \emph{rank} \emph{of the conical local system }is
the rank of the perfect algebra $A_{1}$.
\begin{thm}
\label{thm:Conical of large rank} Every simple locally finite associative
algebra over $\bbF$ has conical local system of arbitrary large rank.
\end{thm}

\begin{proof}
Let $A$ be a simple locally finite associative algebra over $\bbF$.
By Corollary \ref{thm:simple has perfect local}, $A$ has a perfect
local system of arbitrary large rank. Suppose that $\{A_{\alpha}\}_{\alpha\in\Gamma}$
is a perfect local system of $A$ of arbitrary large rank. Fix $A_{\beta}\in\{A_{\alpha}\}_{\alpha\in\Gamma}$.
Since $A_{\beta}$ is finite dimensional associative algebra, there
is a Levi subalgebra $S_{\beta}$ of $A_{\beta}$ such that $A_{\beta}=S_{\beta}\oplus R_{\beta}$,
where $R_{\beta}=\rad A_{\beta}$ is the radical of $A_{\beta}$.
Let $S$ be a simple component of $S_{\beta}$. For all $\gamma\ge\beta$
we denote by $A_{\gamma}^{s}$ to be the ideal of $A_{\gamma}$ that
generated by $S$. Then $A_{\gamma}^{s}$ is the smallest ideal of
$A_{\gamma}$ that contains $S$. Since $(A_{\gamma}^{s})^{2}\subseteq A_{\gamma}^{s}$
is also an ideal of $A_{\gamma}$ that contains $S$, we get that
$(A_{\gamma}^{s})^{2}=A_{\gamma}^{s}$, so $A_{\gamma}^{s}$ is perfect.
Put $A_{1}^{s}=S$ and $\gamma^{s}=\{\gamma\in\gamma\mid\gamma\ge\beta\}\cup\{1\}$.
Put $A^{s}=\cup_{\gamma\in\gamma^{s}}A_{\gamma}^{s}$. Since $A_{\gamma}^{s}\subseteq A_{\gamma'}^{s}$
for all $\gamma$ and $\gamma^{'}\in\gamma^{s}$ with $\gamma\le\gamma'$,
we get that $A^{s}=\dlim A_{\gamma}^{s}$ is an ideal of $A$. But
$A$ is simple, so $A^{s}=A$ and $\{A_{\gamma}^{s}\}_{\gamma\in\gamma^{s}}$
is a perfect local system of $A$. Moreover, we have $\{A_{\gamma}^{s}\}_{\gamma\in\gamma^{s}}$
is a conical local system of $A$ because $\{A_{\gamma}^{s}\}_{\gamma\in\gamma^{s}}$
is a perfect local system with $\gamma^{s}$ containing a minimal
element $1$ that satisfies the conditions (1), (2) and (3) of Definition
\ref{def:Conical}, so it is a local system of arbitrary large rank,
as required.
\end{proof}
\begin{thm}
Every simple locally finite associative algebra over $\bbF$ has $1$-perfect
conical local system .
\end{thm}

\begin{proof}
Let $A$ be a simple locally finite associative algebra over $\bbF$.
By Theorem \ref{thm:1-perfect local system}, $A$ has a $1$-perfect
local system. It remains to follow the same process as in the proof
of Theorem \ref{thm:Conical of large rank}, we get the required results. 
\end{proof}

\section*{Acknowledgement: }

\subsubsection*{I would like to express my appreciation to my Supervisor Dr Alexander
Baranov for his advice and encouragement.}

\end{document}